\documentclass[11pt]{article}

\usepackage{geometry}
\usepackage{braket}
\usepackage{graphicx,graphics}
\usepackage{amsmath,amssymb,latexsym,lscape,amsthm}
\usepackage{xcolor}
\usepackage{lineno,cite}
\usepackage[unicode=true,pdfusetitle,
bookmarks=true,bookmarksnumbered=false,bookmarksopen=false,breaklinks=false,pdfborder={0 0 0},backref=false,colorlinks=false]
{hyperref}
\theoremstyle{plain}

\newtheorem{thm}{Theorem}
\newtheorem{prop}{Proposition}
\newtheorem{lemma}{Lemma}
\newtheorem{cor}{Corollary}
\theoremstyle{definition}

\newtheorem{remark}{Remark}
\usepackage{tikz}
\usepackage{adjustbox}
\usepackage{accents}
\usepackage{enumerate}
\usepackage[shortlabels]{enumitem}
\setlist[enumerate,1]{label={($\roman*$)}}
\usepackage{comment}
\usepackage{appendix}
\usepackage{booktabs,tabularx}
\usepackage{subfig}
\usepackage{comment}
\usepackage{cleveref}
\usepackage{lipsum}
\newcommand{\Rzero}{\mathcal{R}_0}
\newcommand{\Rc}{\mathcal{R}_c}
\usepackage{float}
\definecolor{blue}{rgb}{0.0, 0.0, 1.0}
\definecolor{green}{rgb}{0.0, 0.5, 0.0}
\hypersetup{colorlinks,linkcolor={blue},citecolor={green},urlcolor={blue}}
\newcommand{\R}{\mathbb{R}}

\renewcommand{\arraystretch}{1.15}
\newcommand{\DeclareMathOp}[1]{\csgdef{#1}{\operatorname{#1}}}
\newcommand{\DeclareMathOps}[1]{\forcsvlist{\DeclareMathOp}{#1}}
\DeclareMathOps{ diam, dist, Lip, tr, det, Ran,}

\begin{document}

\begin{center}
    {\Large \bf
    {Global stability and uniform persistence in an epidemic model with saturating fomite-mediated transmission}}
\end{center}

\begin{center}
    Emanuela Penitente\footnote{emanuela.penitente@unina.it},
    Urszula Fory\'s\footnote{urszula@mimuw.edu.pl},
    Burcu G\"urb\"uz\footnote{burcu.gurbuz@uni-mainz.de ($^*$Corresponding author)}$^{,4*}$
\end{center}

\begin{center}
    $^{1}$Department of Mathematics and Applications,
    University of Naples Federico II,\\ via Cintia 26, I--80126,
    Naples, Italy\\
    $^{2}$Institute of Applied Mathematics and Mechanics, University of Warsaw, \\Banacha 2, 02--097, Warsaw, Poland\\

    $^{3}$Institut für Mathematik, Johannes Gutenberg-Universität,\\
    Staudingerweg 9, 55099, Mainz, Germany\\

    $^{4}$Institute for Quantitative and Computational Biosciences (IQCB), Johannes Gutenberg University Mainz, 55128 Mainz, Germany\\
\end{center}

\begin{abstract}
\noindent
We analyse the global dynamics of a Susceptible--Vaccinated--Exposed--Infected--Recovered (SVEIR) epidemic model with demographic turnover, imperfect vaccination, and two transmission routes: direct host-to-host contagion and indirect transmission via contaminated fomites. Indirect transmission is described through an environmental pathogen concentration and a Holling-type dose--response function, accounting for nonlinear incidence at high contamination levels. Threshold conditions separating disease elimination from long-term persistence are expressed in terms of the control reproduction number $\mathcal R_c$, and the classical threshold condition $\mathcal R_c<1$ is derived for the local asymptotic stability of the disease-free equilibrium. For the Holling type~II case, we further obtain an explicit closed-form sufficient condition for the global asymptotic stability of the disease-free equilibrium by applying the Kamgang--Sallet approach for monotone systems with a Metzler infected subsystem. In the absence of vaccination, this criterion recovers the sharp threshold $\mathcal R_0\le 1$ for the global asymptotic stability of the disease-free equilibrium, where $\mathcal R_0$ denotes the basic reproduction number. Conversely, when $\mathcal R_c>1$, we establish uniform persistence of the infection and the existence of at least one endemic equilibrium using persistence theory for semiflows and an acyclicity analysis of the boundary dynamics. Overall, our results quantify the combined impact of vaccination and saturating fomite-mediated transmission on the global behaviour of the model.
\end{abstract}

\noindent {{\bf Keywords}:
mathematical epidemiology, environmental transmission, Holling-type incidence, global stability, uniform persistence.}
\normalsize

%%% \begin{linenumbers}
%% MSC (2020):
%% Primary: 34D23, 92D30
%% Secondary: 37C75, 34C60
%%% MSC-class: 34D23, 92D30, 37C75, 34C60
%%%%
%% arXiv subject class
%% Primary: math.DS (Dynamical Systems)
%% Secondary: math.AP (Analysis of PDEs/ODEs) OR q-bio.PE (Populations and Evolution)
%%%%%
%% Primary: math.DS
%% Secondary: q-bio.PE
%%%%
%% ACM-class
%% G.1.7 Ordinary Differential Equations
%% J.3 Life and Medical Sciences (Biology and genetics)
%%% ACM-class: G.1.7; J.3

\section{Introduction}
\label{sec:intro}
Infectious diseases continue to present a significant public health challenge worldwide, often spreading through a combination of direct person-to-person contagion and indirect exposure via the environment. In recent decades, it has repeatedly been shown that many infectious agents (e.g., respiratory syncytial virus, rhinovirus, norovirus or influenza virus) spread through multiple pathways, combining direct host-to-host contagion with indirect transmission mediated by the environment \cite{CDC_RSV_HowSpreads,atmar2010noroviruses,kraay2018fomite}. Indirect transmission may occur via contaminated water, aerosols, or inanimate objects (fomites), and it can substantially modify both short-term outbreak trajectories and long-term endemic behaviour. From an empirical standpoint, the potential relevance of contaminated surfaces has been discussed for several respiratory and enteric viruses. For example, Boone and Gerba \cite{boone2007significance}  have synthesised the literature on this topic, and Kampf et al. \cite{kampf2020persistence} have conducted an extensive review of the persistence and inactivation of coronaviruses on surfaces. While the quantitative contribution of fomites depends on pathogen survival, contact patterns, and hygiene practices, incorporating an explicit environmental component in transmission models is an established way to assess the epidemiological impact of environmental persistence and interventions.

A common modelling approach involves augmenting classical compartmental frameworks by introducing an additional variable, say $C(t)$, which tracks the concentration of pathogens in an environmental reservoir. This viewpoint has been developed for environmentally mediated transmission in different contexts; for instance, Li et al. \cite{li2009dynamics} proposed a mechanistic formulation linking deposition and pickup of pathogens in the environment to human infection dynamics, and Tien and Earn \cite{tien2010multiple} investigated how multiple transmission routes reshape threshold conditions and outbreak patterns in waterborne disease models. More generally, environmental persistence can create effective delays and feedback loops that blur the distinction between ``direct'' and ``indirect'' transmission when environmental decay is rapid. However, substantial persistence can also lead to qualitatively different dynamics~\cite{breban2013role}. For influenza-like infections, explicit fomite-mediated models have been analysed to quantify when the fomite route alone can sustain transmission and how it interacts with direct transmission~\cite{zhao2012model}.

An important modelling choice concerns the functional form linking the environmental pathogen concentration to the per--susceptible infection probability, often referred to as the \emph{dose--response} relationship \cite{brouwer2017dose}. While adopting a linear dose--response function is analytically convenient and is frequently adopted for mathematical tractability, it may overestimate infection pressure at high contamination levels, where behavioural responses, saturation of contact rates, and nonlinear dose--response effects become relevant. An alternative is provided by the Holling--type functional response, originally introduced in ecology \cite{holling1959components} and now widely used in epidemiological modelling to represent saturating transmission \cite{capasso1978generalization,buonomo2008global,chakroborty2026seirv}. This type of dose--response relationship interpolates between low-dose (approximately linear) exposure and high-dose saturation, and is particularly suitable when the pathogen concentration $C(t)$ represents an aggregate contamination level over many fomites, so that the marginal effect of additional contamination diminishes once most contacts already involve contaminated surfaces.

A Holling--type dose--response function has been adopted, among others, by G{\"o}k{\c{c}}e et al.~\cite{gokcce2024dynamics}, who propose an SVEIR (Susceptible--Vaccinated--Exposed--Infected--Recovered) framework in which vaccinated individuals may remain partially susceptible before acquiring effective immunity. Their formulation is both mathematically tractable and sufficiently general to accommodate different pathogens, while explicitly incorporating vaccination as an intervention, thereby enabling the assessment of how immunisation strategies modulate transmission potential in the presence of an environmental pathway, even when direct contacts are reduced. In their work, the authors perform a local stability and bifurcation analysis, showing that the qualitative behaviour of the model depends on the choice of the dose--response function: in some cases, backward bifurcation may occur, with the possibility of multiple endemic equilibria even when the basic reproduction number is below one. Although~\cite{gokcce2024dynamics} provides a thorough local and bifurcation analysis, complementary results on the global dynamics (e.g., global attractivity or persistence properties) are beyond the scope of that study. However, their findings motivate further investigation of the model's global dynamics, which would complement the local and bifurcation picture developed therein.

Establishing global conditions for disease extinction or persistence is of significant biological importance. In particular, global asymptotic stability of the disease-free equilibrium under vaccination implies, from a biological standpoint, that vaccine-driven elimination is guaranteed to be robust, regardless of the initial magnitude of infection in the population \cite{buonomo2008global}. Nevertheless, proving such results is typically challenging in models incorporating nonlinear incidence or dose–response mechanisms, as the dynamics may admit multiple equilibria and more complex behaviours, and the construction of suitable Lyapunov functions is often technically demanding \cite{ottaviano2022global,shuai2013global}.

Motivated by these considerations, we consider the SVEIR model with Holling-type dose--response proposed in~\cite{gokcce2024dynamics} and investigate its global dynamics, with a focus on establishing rigorous extinction and persistence disease conditions. From a mathematical perspective, the threshold analysis is typically organised around the control reproduction number $\mathcal{R}_c$ (or equivalently, around the basic reproduction number $\Rzero$ when containment interventions are not in place) \cite{van2002reproduction,gumel2004modelling}. In the Holling type~II case, we use the Kamgang-Sallet approach for decompositions with a Metzler infected subsystem and suitable dissipativity properties to obtain a global extinction condition for the disease--free state~\cite{kamgang2008computation}. Conversely, when $\mathcal{R}_c>1$, a key question is whether the infection persists uniformly away from the disease--free boundary. To address this, we adopt persistence theory for semiflows, using boundary dynamics and acyclicity arguments~\cite{thieme1993persistence}.

Altogether, our analysis sheds light on the global dynamical properties of the model and on the interplay between vaccination and saturating fomite-mediated transmission in shaping the long-term epidemic behaviour. The remainder of the paper is organised as follows. In Section \ref{sec: model and basic properties}, we introduce the model and establish its basic analytical properties, including the existence of a biologically feasible region. Section \ref{sec: extinction} is devoted to the analysis of the disease-free equilibrium and its stability properties. In Section \ref{sec: persistence}, we investigate the persistence of the disease and the existence of endemic equilibria. Finally, Section \ref{sec:conclusions} concludes the paper with a discussion of the results and perspectives for future work.

\section{The model and its basic properties}
\label{sec: model and basic properties}

\noindent
We consider an infectious disease transmitted via direct person--to--person contact and via contaminated inanimate objects (fomites), without restricting to any specific pathogen class (e.g., viral or bacterial). The host population is partitioned into five mutually exclusive compartments: susceptible ($S$), exposed ($E$), infectious ($I$), vaccinated ($V$), and recovered/immune individuals ($R$). Here, exposed individuals are infected but neither infectious nor symptomatic; vaccinated individuals have received at least one vaccine dose but have not yet developed full protective immunity; and recovered individuals are fully immune, having acquired permanent immunity either through prior infection or through vaccine-induced protection. The sizes of the five compartments at time $t$, denoted by $S(t)$, $E(t)$, $I(t)$, $V(t)$, and $R(t)$, are state variables of the model. The total population size is
\[
N(t)=S(t)+E(t)+I(t)+V(t)+R(t).
\]

   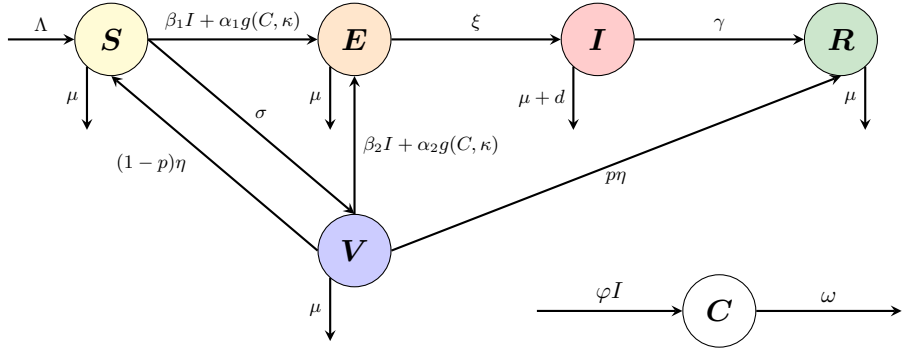
\begin{figure}
            \centering
            \tikzstyle{S} = [
                circle, minimum width=1.2cm, minimum height=1.2cm,
                text centered, draw=black,fill=yellow!20
            ]

            \tikzstyle{E} = [
                circle, minimum width=1.2cm, minimum height=1.2cm,
                text centered, draw=black,fill=orange!20
            ]

            \tikzstyle{I} = [
                circle, minimum width=1.2cm, minimum height=1.2cm,
                text centered, draw=black,fill=red!20
            ]

            \tikzstyle{V} = [
                circle, minimum width=1.2cm, minimum height=1.2cm,
                text centered, draw=black,fill=blue!20
            ]

            \tikzstyle{R} = [
                circle, minimum width=1.2cm, minimum height=1.2cm,
                text centered, draw=black,fill=green!20
            ]

            \tikzstyle{C} = [
                circle, minimum width=1.2cm, minimum height=1.2cm,
                text centered, draw=black,fill=white!20
            ]

            \tikzstyle{arrow} = [thick,->,>=stealth]
            \begin{tikzpicture}[node distance=2cm, scale=0.8, transform shape]
                \node (S) [S] {\Large $\boldsymbol{S}$};
                \node (E) [E, right of=S, xshift=2cm]{\Large $\boldsymbol{E}$};
                \node (I) [I, right of=E, xshift=2cm]{\Large $\boldsymbol{I}$};
                \node (R) [R, right of=I, xshift=2cm]{\Large $\boldsymbol{R}$};
                \node (V) [V, below of=E, yshift=-1.5cm]{\Large $\boldsymbol{V}$};
                \node (C) [C, right of=V, xshift=4cm, yshift=-1cm]{\Large $\boldsymbol{C}$};
                
                \draw [arrow] (S) -- node[anchor=south] {\footnotesize $\beta_1 I + \alpha_1g (C, \kappa)$} (E);
                \draw [arrow] (E) -- node[anchor=south] {\footnotesize $\xi $} (I);
                \draw [arrow] (I) -- node[anchor=south] {\footnotesize $\gamma $} (R);
                \draw [arrow] (V.west) -- node[anchor=east] {\footnotesize $(1- p) \eta \quad$} (S.south);
                \draw [arrow] (S.east) -- node[anchor=south] {\footnotesize $\quad \sigma $} (V.north);
                \draw [arrow] (V.north) -- node[anchor=west] {\footnotesize $\beta_2 I + \alpha_2g (C, \kappa)$} (E.south);
                \draw [arrow] (V.east) -- node[anchor=north] {\footnotesize $p \eta $} (R.south);
                \draw [arrow] (-0.4,-0.45) -- node[anchor=east] {\footnotesize $\mu $} (-0.4,-1.5);
                \draw [arrow] (-1.7, 0) -- node[anchor=south] {\footnotesize $\Lambda$} (S.west);
                \draw [arrow] (3.6,-0.45) -- node[anchor=east] {\footnotesize $\mu $} (3.6,-1.5);
                \draw [arrow] (7.6,-0.45) -- node[anchor=east] {\footnotesize $\mu + d$} (7.6,-1.5);
                 \draw [arrow] (12.4,-0.45) -- node[anchor=east] {\footnotesize $\mu $} (12.4,-1.5);
                 \draw [arrow] (3.6,-3.95) -- node[anchor=east] {\footnotesize $\mu $} (3.6,-5.0);
                 \draw [arrow] (7,-4.5) -- node[anchor=south] {$\varphi I$} (C);
                \draw [arrow] (C) -- node[anchor=south] {$\omega$} (13,-4.5);
            \end{tikzpicture}
            \caption{Flowchart of the model.}
            \label{fig:flowchart}
        \end{figure}

\noindent
The model, described in detail in G\"ok\c{c}e \emph{et al.}~\cite{gokcce2024dynamics}, incorporates demographic turnover through a constant recruitment rate $\Lambda$ into the susceptible class and a per-capita natural death rate $\mu$ acting on all compartments. Transmission operates through two routes: $(i)$ direct contact with infectious individuals and $(ii)$ exposure to the pathogen present in the environment (fomites). Both susceptible and vaccinated individuals are at risk of infection, since the latter have not yet developed full protection. For direct host-to-host transmission, susceptibles acquire infection at per-capita transmission rate $\beta_1$, whereas vaccinated individuals become infected at a reduced transmission rate $\beta_2$.

To represent infection caused by fomites, we introduce an additional state variable $C(t)$, representing the concentration of pathogen in the environment. The impact of environmental contamination on transmission is modelled by the continuous, bounded dose--response function
\begin{equation}\label{eq:C(t)}
  g(C,\kappa)=\frac{C^{\,n}}{C^{\,n}+\kappa^{\,n}}\,, \qquad n\in\mathbb{N},
\end{equation}
\noindent
where $\kappa>0$ is a half--saturation constant and $n$ is fixed. For $n=1$, $g$ reduces to the Holling type~II response, while for $n\ge 2$ it corresponds to a Holling type~III response. In particular, $g(0,\kappa)=0$ and $\lim_{C\to\infty} g(C,\kappa)=1$, so that the environmental contribution saturates at high contamination levels. Moreover,
\[
g'(0,\kappa)=
\begin{cases}
\kappa^{-1}, & n=1,\\[1mm]
0, & n\ge 2,
\end{cases}
\]
which shows that environmental transmission contributes linearly at low contamination levels only in the case $n=1$, whereas for $n\ge 2$ it is of higher order near the origin. This distinction explains why the environmental pathway enters the reproduction numbers only for $n=1$ and may lead to qualitatively different global dynamics, including backward bifurcation for $n=2$, as shown in \cite{gokcce2024dynamics}.

The outflows from the susceptible and vaccinated compartments due to environmental transmission are $\alpha_1 S\,g(C,\kappa)$ and $\alpha_2 V\,g(C,\kappa)$, respectively, where $\alpha_1$ and $\alpha_2$ are the corresponding per-capita transmission coefficients. After exposure, individuals enter the exposed class and progress to the infectious class at a rate $\xi$ (the inverse of the mean latent period). Infectious individuals recover at a rate $\gamma$ or die from the disease at a rate $d$. Vaccination occurs at a rate $\sigma$, moving susceptibles into $V$. Individuals leave the vaccinated compartment at a rate $\eta$ (the inverse of the mean time to develop full immune protection). Upon leaving the compartment $V$, a proportion $p\in(0,1)$ achieves full protection and moves to $R$, while the remaining proportion $1-p$ returns to $S$, since protection fails. We refer to $p$ as the vaccine effectiveness factor.

\begin{table}
    \centering
    \footnotesize
    \renewcommand{\arraystretch}{1}
    \begin{tabular}{
        >{\centering\arraybackslash}p{2cm}
        >{\raggedright\arraybackslash}p{11cm}
    }
        \toprule
        \textbf{Parameter} & \textbf{Description} \\
        \midrule
        $\Lambda$ & Recruitment rate into the susceptible class \\
        $\mu$ & Natural mortality rate\\
        $\beta_1$ & Transmission rate for susceptible individuals through contact with infectious hosts \\
        $\beta_2$ & Transmission rate for vaccinated (not yet fully protected) individuals through contact with infectious hosts \\
        $\alpha_{1}$ & Transmission rate of the pathogen from the environment to susceptible individuals \\
        $\alpha_{2}$ & Transmission rate of the pathogen from the environment to vaccinated individuals \\
        $\gamma$ & Recovery rate  \\
        $d$ & Disease-induced mortality rate \\
        $\xi$ & Progression rate from exposed to infectious \\
        $\sigma$ & Vaccination rate of susceptible individuals \\
        $\varphi$ & Pathogen shedding rate from infectious individuals into the environment \\
        $p$ & Vaccine effectiveness upon immune maturation \\
        $\eta$ & Immune maturation rate \\
        $\omega$ & Environmental pathogen decay rate \\
        \bottomrule
    \end{tabular}
    \caption{ \small Model parameters and their biological interpretation.}
    \label{tab:parameters}
\end{table}

\noindent
The model, whose flowchart is reported in Figure~\ref{fig:flowchart}, is given by the following system of nonlinear ordinary differential equations:
    \begin{align*}
        \dot{S} &= \Lambda - \beta_1 S I - \sigma S + (1-p)\eta V - \alpha_1 S\, g(C,\kappa) - \mu S, \\
        \dot{E} &= \beta_1 S I + \beta_2 V I - \xi E - \mu E + \alpha_1 S g(C,\kappa) + \alpha_2 V g(C,\kappa), \\
        \dot{I} &= \xi E - \gamma I - d I - \mu I, \\
        \dot{V} &= \sigma S - \beta_2 V I - (\eta+\mu) V - \alpha_2 V g(C,\kappa), \\
        \dot{R} &= \gamma I - \mu R + p\eta V, \\
        \dot{C} &= \varphi I - \omega C,
    \end{align*}
    where the upper dots denote the time derivatives. Here, $\varphi$ denotes the rate at which infectious individuals shed pathogen into the environment, and $\omega$ is the environmental decay rate.

We assume that environmental pathogen decays with first-order kinetics at a rate $\omega$, giving exponential decay in the absence of
shedding. Infectious individuals contribute to shedding the pathogen to the environment at a mean rate $\varphi$ per infectious individual,
so the total input is proportional to $I$.
All parameters are positive constants; in particular, $p\in(0,1)$, and we assume the biologically meaningful constraints $\alpha_2 \le \alpha_1$ and $\beta_2 \le \beta_1$.

Since the balance equation for the recovered compartment is uncoupled from the remaining ones, we can restrict the analysis to the following subsystem:
\begin{subequations} \label{eq:sys}
    \begin{align}
    \label{subeqS}
        \dot{S} &= \Lambda-\beta_1 SI - \sigma S+(1-p) \eta V-\alpha_1 S g(C,\kappa)-\mu S, \\
        \label{subeqE}
        \dot{E}&= \beta_1 SI +\beta_2 I V - \xi E-\mu E +\alpha_1 S g(C,\kappa)+ \alpha_2 V g(C,\kappa), \\
        \label{subeqI}
        \dot{I} &= \xi E -\gamma I - d I -\mu I, \\
        \label{subeqV}
        \dot{V} &= \sigma S-\beta_2 I V - ( \eta +\mu) V-\alpha_2 V g(C,\kappa),\\
        \label{subeqC}
        \dot{C} &= \varphi I -\omega C,
    \end{align}
\end{subequations}
with initial conditions
\begin{equation} \label{eq:IC}
    S(0) \geq 0, \quad E(0) \geq 0, \quad I(0) \geq 0, \quad V(0) \geq 0, \quad C(0) \geq 0.
\end{equation}
System \eqref{eq:sys}--\eqref{eq:IC} can be written in vector notation as
\begin{equation*}
    \dot{\mathbf{x}}(t) = \mathbf F\bigl(\mathbf{x}(t)\bigr), \quad \mathbf{x}(0) = \mathbf{x}_0 \in \R^5_+,
\end{equation*}
where $\R^5_+$ is the nonnegative cone of $\R^5$, $\mathbf{x}=(S, E, I, V, C)$ is the vector of the state variables and $\mathbf F(\mathbf{x})$ is the autonomous vector field corresponding to the right-hand side of system~\eqref{eq:sys}. As a preliminary step, we establish the existence of a biologically feasible region for the model.
\begin{prop} \label{thm: invariance Omega}
     The region
     \begin{equation*}
          \Omega = \left\{ \, (S, E, I, V, C) \in \R^5_+ \, \Big| \, S > 0, \; \; S+ E + I + V \leq \frac{\Lambda}{\mu}, \; \;  C \leq \frac{\varphi \Lambda}{\omega \mu} \, \right\}
    \end{equation*}
     is positively invariant and globally attractive for system \eqref{eq:sys}--\eqref{eq:IC}.
 \end{prop}
\begin{proof}
    See Appendix A.
\end{proof}

\section{Elimination of the disease} \label{sec: extinction}
\subsection{The reproduction numbers}
\label{sec:repro}

 The model admits a unique disease--free equilibrium, denoted as \begin{equation*}
     E_0=(S_0,0,0,V_0,0),
 \end{equation*}
where
\begin{equation}\label{eq:DFE}
S_0=\frac{\Lambda(\mu+\eta)}{\mu(\mu+\eta)+\sigma(\mu+p\eta)},
\qquad
V_0=\frac{\Lambda\sigma}{\mu(\mu+\eta)+\sigma(\mu+p\eta)}.
\end{equation}

The local and global stability conditions for the disease--free equilibrium are typically established in terms of a threshold quantity measuring the transmission potential of the infection. In the absence of pre-existing immunity or interventions, this threshold is the well--known \emph{basic reproduction number} $\mathcal R_0$, defined as the expected number of secondary infections generated by an index case
introduced into a wholly susceptible population over its entire infectious period and in the absence of pre-existing immunity or interventions~\cite{diekmann1990definition,heffernan2005perspectives}. When specific mitigation measures (e.g., vaccination or contact reduction) are in place, the corresponding quantity is the \emph{control reproduction number}, usually denoted by $\Rc$~\cite{gumel2004modelling}. By construction, $\Rc \leq \Rzero$. Accordingly, throughout this work we use $\mathcal R_c$ to characterise the stability of the disease--free equilibrium of system~\eqref{eq:sys}. For completeness, we will also refer to the baseline quantity $\mathcal R_0$, which can be recovered from the expression of $\mathcal R_c$ by removing vaccination (i.e., by setting
$\sigma=0$, so that $V_0=0$ and $S_0=\Lambda/\mu$). The expression of the control reproduction number $\mathcal R_c$ in terms of the model parameters has already been derived for this model in
\cite{gokcce2024dynamics} by using the next--generation matrix approach~\cite{diekmann1990definition,van2002reproduction}. We recall here its expression:
\begin{align}\label{eq:Rc}
\mathcal R_c
&=\frac{\xi}{(\xi+\mu)(\mu+\gamma+d)}
\left[\beta_1 S_0+\beta_2 V_0+\delta_{1n}\frac{\varphi}{\kappa \omega}(\alpha_1 S_0+\alpha_2 V_0)\right]\\
\notag
&=\frac{\xi}{\omega(\xi+\mu)(\mu+\gamma+d)}
\frac{\Lambda(\eta+\mu)}{\mu(\sigma+\eta+\mu)+p\eta\sigma}
\left[\omega\beta_1+\frac{\delta_{1n}\alpha_1\varphi}{\kappa}
+\frac{\sigma}{\eta+\mu}\bigg(\beta_2\omega+\frac{\delta_{1n}\alpha_2\varphi}{\kappa}\bigg)\right],
\end{align}
where $\delta_{1n}$ is the Kronecker delta ($\delta_{1n}=1$ if $n=1$ and $\delta_{1n}=0$ otherwise). The expression of the basic reproduction number reads instead
\begin{equation} \label{eq:R0}
\mathcal R_0=\frac{\xi\,\Lambda}{\mu(\xi+\mu)(\gamma+d+\mu)}
\left(\beta_1+\delta_{1n}\frac{\alpha_1\varphi}{\kappa\,\omega}\right).
\end{equation}
Note that $\mathcal R_c$ (or analogously, $\Rzero$) can be written as the sum of two contributions: $\mathcal R_c^{\rm dir}$, associated with direct (person--to--person) transmission, and $\mathcal R_c^{\rm env}$, associated with environmental (fomite--mediated) transmission. Specifically, we write
\[
\Rc=\mathcal R_c^{\rm dir}+\delta_{1n}\,\mathcal R_c^{\rm env},
\]
with
\begin{equation} \label{eq: Rcdir, Rcenv}
    \mathcal R_c^{\rm dir}=\frac{\xi(\beta_1 S_0+\beta_2 V_0)}{(\xi+\mu)(\mu+\gamma+d)},
\qquad
\mathcal R_c^{\rm env}=\frac{\xi\varphi(\alpha_1 S_0+\alpha_2 V_0)}{\kappa \omega(\xi+\mu)(\mu+\gamma+d)}.
\end{equation}
The parameter $\mathcal R_c^{\mathrm{dir}}$ measures the expected number of secondary infections generated through direct contacts by a typical infectious individual introduced into the disease--free population, whereas $\mathcal R_c^{\mathrm{env}}$ quantifies the additional infections arising from the environmental
route. In particular, $\mathcal R_c\ge \mathcal R_c^{\mathrm{dir}}$, and when $n>1$, $\mathcal R_c=\mathcal R_c^{\mathrm{dir}}$. Therefore, when $n>1$, the environmental route does not contribute to the disease invasion.

\begin{prop}
	\label{thm: LAS DFE}
	If $\mathcal{R}_c < 1$, then the disease--free equilibrium $E_0$ of system \eqref{eq:sys} is locally asymptotically stable. If $\mathcal{R}_c > 1$, then $E_0$ is unstable.
\end{prop}
\begin{proof}
    By linearising system \eqref{eq:sys} at the disease--free equilibrium $E_0$, we get the characteristic equation det$(J(E_0)-\lambda \mathrm{I}) = 0$, where
    \begin{equation*}
    J(E_0) =
        \begin{bmatrix}
            J_{11} & 0     & J_{13} & J_{14} & J_{15} \\
            0      & J_{22}& J_{23} & 0      & J_{25} \\
            0      & J_{32}& J_{33} & 0      & 0      \\
            J_{41} & 0     & J_{43} & J_{44} & J_{45} \\
            0      & 0     & J_{53} & 0      & J_{55} \\
        \end{bmatrix},
    \end{equation*}
  and the nonzero entries are given by
    \begin{align*}
        J_{11} &= - ( \sigma + \mu),
        & J_{13} &= -\beta_1 S_0 ,
        & J_{14} &= (1-p) \eta ,
        & J_{15} &= -\frac{\alpha_1 S_0 \delta_{1n}}{\kappa},\\
        J_{22} &= - (\xi + \mu) ,
        & J_{23} &= \beta_1 S_0 + \beta_2 V_0 ,
        & J_{25} &= \frac{( \alpha_1 S_0 + \alpha_2 V_0 ) \delta_{1n}}{\kappa}, & J_{32} &= \xi , \\
        J_{33} &= - (\gamma + d + \mu), &
        J_{41} &= \sigma,
        & J_{43} &= - \beta_2 V_0,
        & J_{44} &= - (\eta + \mu),\\
        J_{45} &= - \frac{\alpha_2 V_0 \delta_{1n}}{\kappa}, & J_{53} &= \varphi ,
        & J_{55} &= -\omega.
    \end{align*}
    By interchanging the second row and column of $J(E_0)$ with the fourth, we get the block--triangular matrix
    \begin{equation*}
    J(E_0) =
        \begin{bmatrix}
            A & 0 \\
            0 & B
        \end{bmatrix},
    \end{equation*}
    with
    \begin{equation*}
        A = \begin{bmatrix}
            J_{11} & J_{14} \\
            J_{41} & J_{44} \\
        \end{bmatrix}, \qquad B=
        \begin{bmatrix}
            J_{33} & J_{32} & 0 \\
            J_{23} & J_{22} & J_{25} \\
            J_{53} & 0 & J_{55}
        \end{bmatrix}.
    \end{equation*}
    Since $\text{tr}(A) = - \sigma - 2 \mu - \eta < 0$ and $\det(A) = \mu ( \mu + \sigma + \eta )+ p \eta \sigma > 0$, the eigenvalues of $A$ have negative real parts. The other eigenvalues of $J(E_0)$ are those of $B$, whose characteristic equation is $\lambda^3 + \mathcal{A}_1 \lambda^2 +  \mathcal{A}_2 \lambda + \mathcal{A}_3 = 0$,
    where
    \begin{align*}
        \mathcal{A}_1 &= - ( J_{22} + J_{33} + J_{55} ), \\
        \mathcal{A}_2 &= J_{22} J_{33} - J_{23} J_{32} + J_{55}(J_{22} + J_{33}) , \\
        \mathcal{A}_3 &= - J_{22} J_{33} J_{55} - J_{32} J_{25} J_{53} + J_{23} J_{32} J_{55}.
    \end{align*}
    From the Routh--Hurwitz criterion \cite{murray2007mathematical,martcheva2015introduction}, the roots of the characteristic equation have negative real parts if and only if the following conditions hold: $(i)\,\mathcal{A}_1 > 0$, $(ii)\,\mathcal{A}_1  \mathcal{A}_2 - \mathcal{A}_3 > 0$ and $(iii) \, \mathcal{A}_3 \, > 0$. If at least one of them is negative, then there is at least one eigenvalue with a positive real part. Condition $(i)$ is satisfied for all the parameter values, since $\mathcal{A}_1 = \xi + 2\mu + \gamma + d + \omega > 0$.
    Condition $(iii)$ can be rewritten as
    \begin{equation*}
        J_{32}( J_{23}J_{55} - J_{25}J_{53}) > J_{22}J_{33}J_{55},
    \end{equation*}
    i.e.,
    \begin{equation} \label{eq: a3r and a3l}
         \xi \left[] \beta_1 S_0 + \beta_2 V_0 + \varphi (\alpha_1 S_0 + \alpha_2 V_0) \frac{\delta_{1n}}{\omega \kappa} \right] < ( \xi + \mu)( \gamma + d + \mu).
    \end{equation}
    Remembering that $V_0 = \frac{\sigma}{\eta + \mu}S_0$, the condition $\mathcal{A}_3 > 0$ can be rewritten as
    \begin{equation*}
     \frac{ \xi S_0 }{( \xi + \mu)( \gamma + d + \mu)} \left[ \beta_1 + \varphi \alpha_1\frac{\delta_{1n}}{\omega \kappa}  + \frac{\sigma}{\eta + \mu} \bigg( \beta_2 + \varphi \alpha_2 \frac{\delta_{1n}}{\omega \kappa} \bigg) \right] < 1,
    \end{equation*}
where the left--hand side is exactly the control reproduction number $\mathcal{R}_c$. Thus, $\mathcal{A}_3 > 0$ if and only if $\mathcal{R}_c < 1$. We show now that condition $(iii)$ implies condition $(ii)$. To this aim, we preliminarily show that $\mathcal{A}_3 > 0$ implies $\mathcal{A}_2 > 0$ and then show that $\mathcal{A}_2 > 0$ and $\mathcal{A}_3 > 0$ imply $\mathcal{A}_1 \mathcal{A}_2 > \mathcal{A}_3$.  Firstly, we denote as $\mathcal{A}_3^L$ and $\mathcal{A}_3^R$ the left--hand and right--hand sides of equation~\eqref{eq: a3r and a3l}, respectively:
    \begin{align*}
        \mathcal{A}_3^L = \xi \left[\beta_1 S_0 + \beta_2 V_0 + \varphi (\alpha_1 S_0 + \alpha_2 V_0) \frac{\delta_{1n}}{\omega \kappa} \right], \qquad
        \mathcal{A}_3^R = ( \xi + \mu)( \gamma + d + \mu).
    \end{align*}
Note that, with this notation, $\mathcal{A}_3 = \omega(\mathcal{A}_3^R - \mathcal{A}_3^L)$ and condition $(iii)$ can be rewritten as
\begin{equation} \label{a3 > 0}
    \mathcal{A}_3 > 0 \quad \iff \quad \mathcal{A}_3^L < \mathcal{A}_3^R.
\end{equation}
    In addition, the inequality $\mathcal{A}_2 > 0$ can be rewritten as $J_{22} J_{33} + J_{55}(J_{22} + J_{33}) > J_{23} J_{32}$, that is,
\begin{equation} \label{eq: a2 > 0}
    (\xi + \mu)( \gamma + d + \mu) + \omega(\xi + 2\mu + \gamma + d) > (\beta_1 S_0 + \beta_2 V_0) \xi.
\end{equation}
By denoting
\begin{align*}
    \mathcal{A}_2^L = (\xi + \mu)( \gamma + d + \mu) + \omega(\xi + 2\mu + \gamma + d),\qquad
    \mathcal{A}_2^R = (\beta_1 S_0 + \beta_2 V_0) \xi,
\end{align*}
we get
\begin{equation}
     \mathcal{A}_2 > 0 \quad \iff \quad \mathcal{A}_2^L > \mathcal{A}_2^R.
\end{equation}

Noting that $\mathcal{A}_3^L \geq  \mathcal{A}_2^R$ (for $n=1$, $\mathcal{A}_3^L > \mathcal{A}_2^R$, while if $n\neq 1$ then $\delta_{1n}=0$ and $\mathcal A_3^L=\mathcal A_2^R$)
 and $\mathcal{A}_2^L > \mathcal{A}_3^R$ yields:

\begin{equation*}
    \mathcal{A}_2^L > \mathcal{A}_3^R > \mathcal{A}_3^L \geq \mathcal{A}_2^R \quad \implies \quad \mathcal{A}_2^L > \mathcal{A}_2^R \quad \implies \mathcal{A}_2 > 0.
\end{equation*}

\noindent Finally, to prove that $\mathcal{A}_1 \mathcal{A}_2 > \mathcal{A}_3$, let us rewrite $\mathcal{A}_1 = \omega + \mathcal{A}_1^*$, with $\mathcal{A}_1^* = \xi + 2\mu + \gamma + d$.
With this notation, the condition $\mathcal{A}_1 \mathcal{A}_2 > \mathcal{A}_3$ becomes
\begin{equation*}
    (\omega + \mathcal{A}_1^*)(\mathcal{A}_2^L - A_2^R) > \omega(\mathcal{A}_3^R - \mathcal{A}_3^L).
\end{equation*}
Since $\mathcal{A}_2^L > \mathcal{A}_3^R$ and $\mathcal{A}_2^R \leq \mathcal{A}_3^L$, then $\omega (\mathcal{A}_2^L - A_2^R) > \omega(\mathcal{A}_3^R - \mathcal{A}_3^L)$. Furthermore, since $\mathcal{A}_1^*(\mathcal{A}_2^L - A_2^R) = \mathcal{A}_1^*\mathcal{A}_2$ is positive if $\mathcal{A}_3 > 0$, the condition $\mathcal{A}_1 \mathcal{A}_2 > \mathcal{A}_3$ is satisfied. This proves that if $\mathcal{R}_c < 1$, then conditions $(i), (ii)$ and $(iii)$ are fulfilled and all the eigenvalues of the matrix $B$ have negative real parts, implying that $E_0$ is locally asymptotically stable when $\mathcal{R}_c < 1$. Conversely, if $\mathcal{R}_c >1$, then $\mathcal A_3 <0$, and there is at least one eigenvalue with positive real part. Thus, if $\mathcal{R}_c >1$, then $E_0$ is unstable.
\end{proof}

\subsection{The global stability result} \label{sec: GAS}

To obtain a sufficient condition for the global stability of the disease--free equilibrium, we follow the approach developed by Kamgang and Sallet~\cite{kamgang2008computation}. Accordingly, system \eqref{eq:sys} can be rewritten in the pseudo--triangular form
\begin{equation}\label{eq:model_matrices}
\begin{cases}
\dot{x}_1 = A_1(x)\,(x_1-x_1^*) + A_{12}(x)\,x_2,\\
\dot{x}_2 = A_2(x)\,x_2,
\end{cases}
\end{equation}
where $x_1=(S,V)^{T}$ collects the \emph{uninfected} compartments and $x_2=(E,I,C)^{T}$ the
\emph{infected} ones. Moreover, $A_2(x)$ is a Metzler matrix, i.e., its off--diagonal entries satisfy
$a_{ij}(x)\ge 0$ for $i\neq j$. The global asymptotic stability of the disease--free equilibrium $x^*=(x_1^*,0)$ follows from Theorem~4.3 in
\cite{kamgang2008computation}, provided that: $(i)$ the system is point dissipative on $\Omega$, $(ii)$ the
disease--free subsystem is globally asymptotically stable at $x_1^*$, and
$(iii)$ suitable spectral and upper--bound conditions on $A_2(x)$ hold (see Appendix B for all conditions). We start by proving point dissipativity on
$\Omega$.

\begin{prop} \label{lemma: point dissipativity}
    System \eqref{eq:sys} is point dissipative on $\Omega$, i.e., there exists a compact set $K \subseteq \Omega$ such that for every $y \in \Omega$, there exists a time $t(y)$ for which $x(t, 0, y) \in \mathring{K}$ for every $t \geq t(y)$. The compact set $K$ is given by
     \begin{equation*}
        K = \left\{ (S, E, I, V, C) \in \R^5_+ \, \Big| \, S \geq \delta, \, S + E + I + V \leq \frac{\Lambda}{\mu}, \, C \leq \frac{\varphi \Lambda}{\omega \mu} \right\}\,,
    \end{equation*}
where
\begin{equation} \label{eq:delta}
\delta = \frac{1}{2} \frac{\Lambda \mu}{\beta_1 \Lambda + (\sigma + \alpha_1 + \mu) \mu}\,.
\end{equation}
\end{prop}

\begin{proof}
    Our goal is to prove that for every $S(0) > 0$, there exists a $t_{S(0)} > 0$ such that $S(t) > \delta$ for every  $t > t_{S(0)}$. We preliminarily prove the statement for $S(0) > \delta$. By contradiction suppose that for every $t_{S(0)}$, a $\bar{t} > t_{S(0)}$ exists such that $S(\bar{t}) \leq \delta$.
Since $S(t)$ is continuous on its domain and $S(0) > \delta$, there also exists a $t_{S(0)} < \hat{t} < \bar{t}$ such that $S(\hat{t}) = \delta$. If the function $S(t)$ is monotone and non--increasing, we have that $\dot{S}(\hat{t}) \leq 0$. However, even if $S(t)$ is not monotone and $\dot{S}(\hat{t}) > 0$, there surely exists another $\hat{t}'$, with $ \hat{t} < \hat{t}' < \bar{t}$, such that $S(\hat{t}') = \delta$ and
$\dot S(\hat{t}')\le 0$.
For this reason, without loss of generality, we can assume that:
\begin{gather}
    \label{1 cond}
    S(\hat{t}) = \delta\,, \\
    \label{2 cond}
    \dot{S}(\hat{t}) \leq 0 \,.
\end{gather}
From equation \eqref{subeqS}, using inequality \eqref{1 cond}, that $g (C, \kappa) \leq 1$ and $I \leq \Lambda/ \mu$, we have:
\begin{align*}
    \dot{S}(\hat{t}) &= \Lambda + (1- p) \eta V - \delta \big(\beta_1 I + \sigma + \alpha_1 g(C, \kappa) + \mu\big) \\
    &\geq \Lambda - \delta \left(\beta_1 \frac{\Lambda}{\mu} + \sigma + \alpha_1 + \mu \right) \geq \frac{\Lambda}{2} > 0 \,,
\end{align*}
in contradiction with \eqref{2 cond}. If $0 < S(0) \leq \delta$, using that $I(0) < \Lambda / \mu $ and $g(C(0), \kappa)<1$ then:
\begin{align*}
    \dot{S}(0) &= \Lambda + (1- p) \eta V(0) - S(0) \big(\beta_1 I(0) + \sigma + \alpha_1 g(C(0), \kappa) + \mu\big)\\
    & \geq \Lambda - \delta \left(\beta_1 \frac{\Lambda}{\mu} + \sigma + \alpha_1 + \mu \right)  \geq \frac{\Lambda}{2} > 0 \,.
\end{align*}
Therefore, $S(t)$ is initially increasing. As long as $S(t) < \delta$, the derivative $\dot{S}(t)$ remains positive and continues to grow until it reaches the value $S(t) = \delta$. At this point, the previous argument can be reapplied to complete the proof.
\end{proof}

 \begin{thm}
        Let $n=1$ in \cref{eq:C(t)}. The disease--free equilibrium $E_0$ of system \eqref{eq:sys} is globally asymptotically stable in the closure $\bar{\Omega}$ whenever $\mathcal{J}_c\leq 1$, where
    \begin{equation} \label{J0 GAS}
          \mathcal{J}_c = \frac{\Lambda}{\mu}\frac{\xi}{(\xi + \mu)(\gamma + d + \mu)} \left(  \frac{\varphi \alpha_1}{\kappa \omega}  +  \beta_1 \right)\,.
      \end{equation}
    \end{thm}

    \begin{proof}
      Condition \textbf{\emph{A1}} of Theorem 4.3 in \cite{kamgang2008computation} has been proved in Proposition \ref{lemma: point dissipativity}.  The disease--free subsystem mentioned in condition \textbf{\emph{A2}} is linear--affine and reads
       \begin{equation} \label{eq: SVR system}
               \begin{bmatrix}
\dot S\\
\dot V
\end{bmatrix}
=
\begin{bmatrix}
-(\sigma+\mu) & (1-p)\eta\\
\sigma & -(\eta+\mu)
\end{bmatrix}
\begin{bmatrix}
S\\[2pt]
V
\end{bmatrix}
+
\begin{bmatrix}
\Lambda\\
0
\end{bmatrix}.
       \end{equation}
   Its equilibrium $E_0=(S_0, V_0)$ is globally asymptotically stable, since the matrix in \eqref{eq: SVR system} has negative trace and positive determinant. Thus, condition~\textbf{\emph{A2}} is satisfied. The matrix $A_2(x)$ reads
    \begin{equation}
        A_2(S, V, C) = \begin{bmatrix}
            - (\xi + \mu) & \beta_1 S + \beta_2 V & \dfrac{\alpha_1 S + \alpha_2 V}{C + \kappa} \\
            \xi & - ( \gamma + d + \mu) & 0 \\
            0 & \varphi & - \omega
        \end{bmatrix}.
    \end{equation}
    This matrix is Metzler and irreducible for all $(S, V, C) \in \Omega$, since its associated directed graph is strongly connected. Thus, the condition \textbf{\emph{A3}} is satisfied. We must determine the upper bound matrix $\bar{A}_2$.  Since for every $(S, V, C) \in \Omega$ it holds that
    \begin{equation*}
        S + V \leq \frac{\Lambda}{\mu},  \qquad \frac{1}{C + \kappa} \leq \frac{1}{ \kappa},
    \end{equation*}
    and $\alpha_2\leq\alpha_1$, $\beta_2\leq\beta_1$,
    the following upper bounds can be used for the terms $(A_2)_{12}$ and $(A_2)_{13}$:
    \begin{equation*}
        \beta_1 S + \beta_2 V \leq \beta_1 (S+V) \leq \beta_1\frac{\Lambda}{\mu} ,
        \qquad
        \frac{\alpha_1 S + \alpha_2 V}{C + \kappa} \leq
    \frac{\alpha_1(S+V)}{\kappa}\leq \frac{\alpha_1}{\kappa}
        \frac{\Lambda}{\mu } .
    \end{equation*}
The upper--bound matrix is then
    \begin{equation} \label{upp bound matrix}
    \bar{A}_2 =
        \begin{bmatrix}
        - ( \xi + \mu) & \dfrac{\beta_1\Lambda}{\mu} & \dfrac{\alpha_1\Lambda}{\mu \kappa} \\
        \xi & - (\gamma + d + \mu) & 0 \\
        0 & \varphi & - \omega
    \end{bmatrix}.
    \end{equation}
    If $\alpha_1\neq \alpha_2$ or $\beta_1 \neq \beta_2$, then this upper--bound is  attained in $\Omega$ at the unique point $\left(\Lambda/\mu, 0 , 0 , 0 , 0 \right)$. If $\alpha_1=\alpha_2$ and $\beta_1=\beta_2$, then it is attained at infinitely many points satisfying $S+V=\Lambda/\mu$ with other coordinates equal to 0. However, in both cases, it is realised for the points in the disease-free submanifold, which implies that the condition \textbf{\emph{A4}} is also satisfied. On the other hand, it is not realised for the Jacobian at the DFE, so we shall obtain only a sufficient condition.  We now compute the eigenvalues of the matrix $\bar{A}_2$ to derive a condition for the stability modulus $s(\bar{A}_2)$ to be non--positive and satisfy the hypothesis \textbf{\emph{A5}}.
By denoting
\begin{align*}
       q_1 = \xi + \mu, \qquad q_2 = \frac{\beta_1\Lambda}{\mu},\qquad
       q_3 =\frac{\alpha_1\Lambda}{\mu \kappa}, \qquad q_4= \gamma + d + \mu,
   \end{align*}
   the characteristic polynomial of $\bar{A}_2$ is
   \begin{equation*}
    p(\lambda) = \lambda^3 + \mathcal{A}_1\lambda^2 + \mathcal{A}_2\lambda + \mathcal{A}_3 = 0
   \end{equation*}
   with
      \begin{align*}
         \mathcal A_1=q_1+q_4+\omega,
\qquad
\mathcal A_2=q_1q_4+\omega(q_1+q_4)-\xi q_2,
\qquad
\mathcal A_3=\omega(q_1q_4-\xi q_2)-\xi\varphi q_3.
      \end{align*}
     From the Routh--Hurwitz criterion \cite{murray2007mathematical,martcheva2015introduction}, all the roots of the characteristic equation
\(
p(\lambda) =0
\)
have negative real parts if and only if
\[
(i)\quad \mathcal A_i>0,\ i=1,2,3,
\qquad\text{and}\qquad
(ii)\quad \mathcal A_1\mathcal A_2>\mathcal A_3.
\]
Clearly, \(\mathcal A_1=q_1+q_4+\omega>0\).
Moreover, the condition \(\mathcal A_3>0\) can be rewritten as
\[
\omega\big(q_1q_4-\xi q_2\big)-\xi\varphi q_3>0
\quad\Longleftrightarrow\quad
\frac{\xi(\varphi q_3+\omega q_2)}{\omega q_1q_4}<1,
\]
that is, \(\mathcal J_c<1\), where \(\mathcal J_c\) is defined in \eqref{J0 GAS}. It remains to verify that \(\mathcal A_2>0\) under the same assumption.
Assume \(\mathcal J_c<1\) (or equivalently, \(\mathcal A_3>0\)). Then \[
\omega\big(q_1q_4-\xi q_2\big)>\xi\varphi q_3
\quad\Longrightarrow\quad
q_1q_4-\xi q_2>\frac{\xi\varphi q_3}{\omega}>0.
\]
In particular, \(q_1q_4-\xi q_2>0\), and therefore
\[
\mathcal A_2=q_1q_4+\omega(q_1+q_4)-\xi q_2
=(q_1q_4-\xi q_2)+\omega(q_1+q_4)>0.
\]
Hence, whenever \(\mathcal J_c<1\), condition \((i)\) of the Routh--Hurwitz criterion is satisfied.

    Finally, still under the assumption $\mathcal A_3>0$, a direct computation gives
\[
\mathcal A_1\mathcal A_2-\mathcal A_3
=(q_1+q_4)\left[(q_1q_4-\xi q_2)+\omega(q_1+q_4+\omega)\right]+\xi\varphi q_3,
\]
which is strictly positive since $q_1+q_4>0$, $\omega>0$, $\xi\varphi q_3>0$, and
$q_1q_4-\xi q_2>\xi\varphi q_3/\omega>0$ follows from $\mathcal A_3>0$.
Therefore, whenever $\mathcal J_c<1$ (equivalently, $\mathcal A_3>0$), condition $(ii)$ also holds and the Routh--Hurwitz criterion is satisfied, yielding $s(\bar A_2)<0$ and then the global stability of $E_0$.

Finally, when $\mathcal J_c=1$ (equivalently, $\mathcal A_3=0$), we have $p(0)=\mathcal A_3=0$ and
the characteristic polynomial can be written as
\[
p(\lambda)=\lambda\left(\lambda^2+\mathcal A_1\lambda+\mathcal A_2\right).
\]
Since $\mathcal A_1>0$ and, under $\mathcal A_3=0$, one obtains
$q_1q_4-\xi q_2=\xi\varphi q_3/\omega>0$, it follows that
\[
\mathcal A_2=(q_1q_4-\xi q_2)+\omega(q_1+q_4)>0.
\]
Therefore, $\lambda=0$ is a simple eigenvalue and the remaining two eigenvalues have negative real
parts, so that $s(\bar A_2)=0$ when $\mathcal J_c=1$. In particular,
\[
s(\bar A_2)\le 0 \quad \text{whenever}\quad \mathcal J_c\le 1.
\]
Consequently, hypothesis~\textbf{\emph{A5}} is satisfied for $\mathcal J_c\le 1$. Since hypotheses
\textbf{\emph{A1}}--\textbf{\emph{A5}} have been verified above, Theorem~4.3 yields that the disease--free equilibrium $E_0$ is globally asymptotically stable in $\bar \Omega$ for $\mathcal J_c\le 1$.
    \end{proof}

\begin{remark}
    For $n=1$, the control reproduction number $\mathcal R_c$ reads
\[
\mathcal R_c
=\frac{\xi}{(\xi+\mu)(\gamma+d+\mu)}
\left[\beta_1 S_0+\beta_2 V_0
+\frac{\varphi}{\kappa\omega}(\alpha_1 S_0+\alpha_2 V_0)\right],
\]
whereas the threshold $\mathcal J_c$ is given by
\[
\mathcal J_c
=\frac{\Lambda}{\mu}\frac{\xi}{(\xi + \mu)(\gamma + d + \mu)} \left(  \frac{\varphi \alpha_1}{\kappa \omega}  +  \beta_1 \right).
\]
Since at the disease--free equilibrium it holds that  $S_0 +V_0 < \Lambda/\mu$, we obtain the estimate
\[
\beta_1 S_0+\beta_2 V_0
+\frac{\varphi}{\kappa\omega}(\alpha_1 S_0+\alpha_2 V_0)
\leq
\beta_1 (S_0+V_0)
+\frac{\alpha_1\varphi}{\kappa\omega}(S_0+V_0)
<
\frac{\Lambda}{\mu}
\left(
\beta_1
+\frac{\alpha_1\varphi}{\kappa\omega}
\right),
\]
and thus, $\mathcal R_c < \mathcal J_c$. In particular, $\mathcal J_c<1$ implies $\mathcal R_c<1$, so that
$\mathcal J_c<1$ is a sufficient condition for disease extinction. The quantity $\mathcal J_c$ arises from replacing the variables $S$ and $V$ in the infected subsystem by their common maximal possible value, $\Lambda/\mu$, in the invariant set $\Omega$, and therefore provides an upper bound on the linearised infection growth.
\end{remark}

As discussed in Subsection~\ref{sec:repro}, in the absence of vaccination, the relevant threshold is
the basic reproduction number $\mathcal R_0$. For $n=1$ it is given by
\[
\mathcal R_0=\frac{\xi\,\Lambda}{\mu(\xi+\mu)(\gamma+d+\mu)}
\left(\beta_1+\frac{\alpha_1\varphi}{\kappa\,\omega}\right).
\]
Clearly, now $\mathcal{R}_0=\mathcal{J}_c$, and in this setting, the sufficient condition obtained via the Kamgang--Sallet approach
coincides with the classical threshold, yielding a necessary and sufficient criterion for global
stability of the disease--free equilibrium.

\begin{cor}\label{novax_GAS}
Let $n=1$ and consider the reduced model obtained in the absence of vaccination. Then, its
disease--free equilibrium
\[
E_0=\left(\frac{\Lambda}{\mu},0,0,0\right)
\]
is globally asymptotically stable in the positively invariant set $\bar\Omega$ if and only if
$\mathcal R_0\le 1$, and is unstable if $\mathcal R_0>1$.
\end{cor}

\section{Persistence of the disease and endemic equilibrium}
\label{sec: persistence}

In this section, we show that the threshold condition $\mathcal{R}_c>1$ guarantees that system~\eqref{eq:sys} admits at least one endemic equilibrium and that the disease \emph{persists} in the population. Persistence means that solutions with a positive amount of infected individuals remain uniformly away from the disease--free boundary, i.e., the number of infected individuals is bounded away from zero for large times. To establish persistence, we follow the approach introduced in a seminal paper by Thieme~\cite{thieme1993persistence}, which relies on the dynamical behaviour of the semiflow near the boundary of a suitable positively invariant set. To prove the main persistence result, we preliminarily prove the following lemma.

\begin{lemma}\label{lem:Mpartial_equals_Sigma}
Let $\Phi$ be the semiflow generated by system~\eqref{eq:sys} on $\bar\Omega$ and let
\[
X_0=\{x\in\bar\Omega:\ E>0,\ I>0,\ C>0\},\qquad
\partial X_0=\{x\in\bar\Omega:\ EIC=0\}.
\]
Define the following subsets of $\partial X_0$:
\[
M_{\partial}
=\Bigl\{x_0\in \partial X_0:\ \Phi_t(x_0)\in \partial X_0 \ \text{for all } t\ge 0\Bigr\},
\qquad
\Sigma=\{x\in\bar\Omega:\ E=I=C=0\}.
\]
Then $M_{\partial}=\Sigma$.
\end{lemma}

\begin{proof}
Clearly, $\Sigma\subseteq M_{\partial}$ since $\Sigma$ is forward invariant. To prove the equality, we show that $(\partial X_0\setminus \Sigma)\cap M_{\partial}=\emptyset$.
Let $x_0\in \partial X_0\setminus \Sigma$. Then at least one among $E(0),I(0),C(0)$ is positive, and at least one among them is zero. Moreover, by Proposition~\ref{thm: invariance Omega}, $S(t)>0$ for all $t>0$. We show that $x_0\notin M_{\partial}$ by proving that the
trajectory leaves $\partial X_0$, i.e., there exists $t^*>0$ such that
$E(t^*)>0$, $I(t^*)>0$, and $C(t^*)>0$.

\medskip
\noindent\textbf{Case 1:} $I(0)>0$ (and $E(0)=0$ and/or $C(0)=0$).
From \eqref{subeqI} we have
\[
\dot I(t)=\xi E(t)-(\gamma+d+\mu)I(t)\ge-(\gamma+d+\mu)I(t),
\]
hence $I(t)\ge I(0)e^{-(\gamma+d+\mu)t}>0$ for all $t\ge 0$.
Then, by variation of constants applied to \eqref{subeqC},
\[
C(t)=e^{-\omega t}C(0)+\varphi\int_0^t e^{-\omega(t-s)}I(s)\,ds>0
\qquad \forall\,t>0.
\]
Finally, since $S(t)>0$ and $I(t)>0$ for $t>0$, from \eqref{subeqE} we obtain
\[
\dot E(t)\ge \beta_1 S(t)I(t)-(\xi+\mu)E(t),
\]
which implies $E(t)>0$ for all $t>0$ by a comparison argument. Hence $\Phi_t(x_0)\notin\partial X_0$
for all sufficiently small $t>0$.

\medskip
\noindent\textbf{Case 2:} $E(0)>0$ (and $I(0)=0$ and/or $C(0)=0$).
If $I(0)>0$, then Case~1 applies. Thus, assume $I(0)=0$. From \eqref{subeqI} we have
$\dot I(0)=\xi E(0)>0$, and therefore $I(t)>0$ for all sufficiently small $t>0$.
Then Case~1 applies and yields $C(t)>0$ and $E(t)>0$ for $t>0$. Consequently, the trajectory exits
$\partial X_0$.

\medskip
\noindent\textbf{Case 3:} $C(0)>0$ (and $I(0)=0$ and/or $E(0)=0$).
If $I(0)>0$ or $E(0)>0$, then Case~1 or Case~2 applies. Thus, assume $I(0)=E(0)=0$.
Since $C(0)>0$, by continuity $C(t)>0$ for $t$ in a right neighbourhood of $0$.
Using \eqref{subeqE}, together with the fact that $S(t)>0$ for $t>0$ and also $g(C(t),\kappa)>0$ whenever $C(t)>0$, we obtain
$\dot E(t)>0$ for $t$ close to $0$, hence $E(t)>0$ for some small $t>0$.
Then \eqref{subeqI} implies $I(t)>0$ for small $t>0$, and Case~1 applies, so the trajectory exits
$\partial X_0$.

\medskip
In all the cases, there exists a $t^*>0$ such that $\Phi_{t^*}(x_0)\notin \partial X_0$. This implies that if $x_0 \in \partial X_0 \setminus \Sigma$, then $x_0\notin M_{\partial}$. Thus, $(\partial X_0 \setminus \Sigma) \cap M_{\partial} = \emptyset$, and hence
$M_{\partial}=\Sigma$.
\end{proof}

\begin{lemma}  \label{lemma: E0 weak repeller for X0}
    If $\mathcal{R}_c>1$, then the disease--free equilibrium $E_0$ of system~\eqref{eq:sys} is a weak repeller for the set \[
X_0=\{x\in\bar\Omega:\ E>0,\ I>0,\ C>0\},
\]
i.e.,
\[
\limsup_{t\to\infty} d\bigl(\Phi_t(x_0),E_0\bigr)>0, \qquad \forall x_0 \in X_0.
\]
\end{lemma}

\begin{proof}
    %\alert{[to be revised]}.
    It is enough to show that $W^s(E_0) \cap X_0 = \emptyset$ when $\mathcal{R}_c>1$, where $W^s(E_0)$ denotes the stable manifold of $E_0$. By contradiction, assume that there exists a solution $\Phi_t(x_0) \in X_0$ for $t \geq 0$ with initial value $x_0 \in X_0$ such that
    \[
    \lim_{t\to \infty} \Phi_t(x_0) = E_0.
    \]
    Namely, for any fixed $\varepsilon >0$, there exists $t_1 >0$ such that for all $t \geq t_1$,
    \begin{align}
S_0-\varepsilon<S(t)<S_0+\varepsilon,\qquad & 0<E(t)<\varepsilon,\quad 0<I(t)<\varepsilon, \label{eq:bounds1}\\
V_0-\varepsilon<V(t)<V_0+\varepsilon,\qquad & 0<C(t)<\varepsilon, \label{eq:bounds2}
\end{align}
    where $S_0$ and $V_0$ are the disease--free equilibrium coordinates given in equation \eqref{eq:DFE}.

    \medskip
\noindent\textbf{Case 1: $n\ge 2$.}
In this case, the function $g$ is such that $g'(0)=0$, and the control reproduction number $\Rc$ reduces to its direct transmission component, i.e., $\mathcal R_c=\mathcal
R_c^{\rm dir}$. Using the model equations \eqref{subeqE}--\eqref{subeqI}, the bounds $S(t)\ge S_0-\varepsilon$, $V(t)\ge V_0-\varepsilon$
and the fact that $g(C,\kappa)\ge0$, we obtain that for $t \geq t_1$ holds:
\begin{align}\label{eq:EI_compare_nge2}
\dot E(t) &\ge \left[\beta_1(S_0-\varepsilon)+\beta_2(V_0-\varepsilon)\right]\,I(t)-(\xi+\mu)E(t),\\
\dot I(t) &= \xi E(t)-(\gamma+d+\mu)I(t).
\end{align}
Let $y(t):=(E(t),I(t))^t$. System \eqref{eq:EI_compare_nge2} can be written in vector form as
\[
\dot y(t)\ge B_\varepsilon\,y(t),\qquad t\ge t_1,
\]
where
\[
B_\varepsilon=
\begin{bmatrix}
-(\xi+\mu) & \beta_1(S_0-\varepsilon)+\beta_2(V_0-\varepsilon)\\[1mm]
\xi & -(\gamma+d+\mu)
\end{bmatrix}.
\]
Since $\mathcal R_c=\mathcal R_c^{\rm dir}>1$, the disease--free equilibrium is unstable (see Proposition~\ref{thm: LAS DFE}), which is
equivalent to $s(B_0)>0$ for
\[
B_0=
\begin{bmatrix}
-(\xi+\mu) & \beta_1 S_0+\beta_2 V_0\\[1mm]
\xi & -(\gamma+d+\mu)
\end{bmatrix}.
\]
Since the spectrum of a matrix depends continuously on its entries \cite{li2019eigenvalue}, there exists $\varepsilon_0>0$ such that
$s(B_\varepsilon)>0$ for all $\varepsilon\in(0,\varepsilon_0]$.
Fix $\varepsilon\in(0,\varepsilon_0]$ and set $\lambda:=s(B_\varepsilon)>0$.
Since $B_\varepsilon$ is Metzler and irreducible, Perron--Frobenius theory yields a vector $v\gg0$ such that
$B_\varepsilon v=\lambda v$ (see, e.g., \cite[Corollary~3.2]{smith1995monotone}).
Choose
\[
\tilde\epsilon:=\min\Big\{\frac{E(t_1)}{v_1},\frac{I(t_1)}{v_2}\Big\}>0,
\qquad\text{so that}\qquad
\tilde\epsilon v\le y(t_1)\ \text{componentwise}.
\]
Let $z(t)$ be the solution of $\dot z=B_\varepsilon z$ with $z(t_1)=\tilde \epsilon v$. Since $\dot y\ge B_\varepsilon y$
and $\dot z=B_\varepsilon z$ is cooperative, the comparison principle gives $y(t)\ge z(t)$ for all $t\ge t_1$.
Moreover, $z(t)=\tilde\epsilon e^{\lambda(t-t_1)}v$, hence $z(t)\to\infty$ as $t\to\infty$, so $y(t)$ cannot
converge to $0$, contradicting $\Phi_t(x_0)\to E_0$.

\medskip
\noindent\textbf{Case 2: $n=1$.}
Here $g(C,\kappa)=C/(C+\kappa)$. Using \eqref{eq:bounds1}--\eqref{eq:bounds2}, for all $t\ge t_1$,
\[
g(C(t),\kappa)=\frac{C(t)}{C(t)+\kappa}\ \ge\ \frac{1}{\kappa+\varepsilon}\,C(t).
\]
Let $y(t):=(E(t),I(t),C(t))^t$. Using \eqref{subeqE}--\eqref{subeqC}, the bounds $S(t)\ge S_0-\varepsilon$,
$V(t)\ge V_0-\varepsilon$ and the above inequality for $g$, we obtain for $t\ge t_1$:
\[
\dot y(t)\ge \widetilde B_\varepsilon\,y(t),
\]
where
\[
\widetilde B_\varepsilon=
\begin{bmatrix}
-(\xi+\mu) &
\beta_1(S_0-\varepsilon)+\beta_2(V_0-\varepsilon) &
\dfrac{\alpha_1(S_0-\varepsilon)+\alpha_2(V_0-\varepsilon)}{\kappa+\varepsilon}\\[2mm]
\xi & -(\gamma+d+\mu) & 0\\
0 & \varphi & -\omega
\end{bmatrix}.
\]
The matrix $\widetilde B_\varepsilon$ is Metzler and irreducible.
Since $\mathcal R_c>1$, the disease--free equilibrium is unstable (Proposition~\ref{thm: LAS DFE}), which is equivalent to
$s(\widetilde B_0)>0$, where $\widetilde B_0=\widetilde B_\varepsilon|_{\varepsilon=0}$.
By continuity, there exists $\varepsilon_1>0$ such that $s(\widetilde B_\varepsilon)>0$ for all
$\varepsilon\in(0,\varepsilon_1]$. Fix such an $\varepsilon$ and set $\lambda:=s(\widetilde B_\varepsilon)>0$.
By Perron--Frobenius theory, there exists $v\gg0$ such that $\widetilde B_\varepsilon v=\lambda v$.
Choose
\[
\tilde\epsilon:=\min\Big\{\frac{E(t_1)}{v_1},\frac{I(t_1)}{v_2},\frac{C(t_1)}{v_3}\Big\}>0,
\qquad\text{so that}\qquad
\tilde \epsilon v\le y(t_1).
\]
Let $z(t)$ solve $\dot z=\widetilde B_\varepsilon z$ with $z(t_1)=\tilde \epsilon v$.
By the comparison principle for cooperative systems, $y(t)\ge z(t)$ for all $t\ge t_1$.
Since $z(t)=\tilde \epsilon e^{\lambda(t-t_1)}v$, we have $z(t)\to\infty$ as $t\to\infty$, hence $y(t)$ cannot
converge to $0$, again contradicting $\Phi_t(x_0)\to E_0$.

\medskip
In both cases, we obtain a contradiction. Therefore $W^s(E_0)\cap X_0=\emptyset$, and $E_0$ is a weak
repeller for $X_0$.

\end{proof}

\begin{thm}
Assume that $\mathcal{R}_c>1$. Then system~\eqref{eq:sys} is uniformly persistent, namely, there exists a constant
$\delta>0$ such that every solution
\(
\Phi_t(x_0)=\bigl(S(t),E(t),I(t),V(t),C(t)\bigr)
\)
of~\eqref{eq:sys} satisfies
\[
\liminf_{t\to\infty} S(t)\ge \delta,\quad
\liminf_{t\to\infty} E(t)\ge \delta,\quad
\liminf_{t\to\infty} I(t)\ge \delta,\quad
\liminf_{t\to\infty} V(t)\ge \delta,\quad
\liminf_{t\to\infty} C(t)\ge \delta,\quad
\]
for every initial condition
\[
x_0=\bigl(S(0),E(0),I(0),V(0),C(0)\bigr)\in \bar{\Omega}
\quad\text{with}\quad E(0)+I(0) + C(0)>0.
\]
Moreover, system~\eqref{eq:sys} admits at least one endemic equilibrium.
\end{thm}

\begin{proof}
By Proposition~\ref{lemma: point dissipativity}, there exist a constant $\delta_S>0$ (given in
\eqref{eq:delta}) and a time $t_{S(0)}>0$ such that
\[
S(t)\ge \delta_S \qquad \text{for all } t\ge t_{S(0)},
\]
where $\delta_S$ is independent of the initial condition. In particular,
$\liminf_{t\to\infty}S(t)\ge \delta_S$. Moreover, since  $I(t)\le \Lambda/\mu$ and $0\le g(C,\kappa)\le 1$, we obtain from equation \eqref{subeqV} that
\begin{align*}
    \dot V(t)
=& \, \sigma S(t)-\Bigl(\eta+\mu+\beta_2 I(t)+\alpha_2 g(C(t),\kappa)\Bigr)V(t) \\
 \ge &\,  \sigma\delta_S-\Bigl(\eta+\mu+\beta_2\frac{\Lambda}{\mu}+\alpha_2\Bigr)V(t), \qquad t\ge t_{S(0)}.
\end{align*}

\noindent By comparison with the linear equation $\dot y=\sigma\delta_S-k y$, where
$k:=\eta+\mu+\beta_2\dfrac{\Lambda}{\mu}+\alpha_2$, it follows that
\[
\liminf_{t\to\infty}V(t)\ge \delta_V:=\frac{\sigma\delta_S}{\eta+\mu+\beta_2\dfrac{\Lambda}{\mu}+\alpha_2}>0.
\]

\medskip
\noindent
Hence, it remains to prove that $E(t)$, $I(t)$ and $C(t)$ are uniformly bounded away from zero
whenever $E(0)+I(0)+C(0)>0$. Let
\begin{align*}
    X_0 &= \left\{ (S, E, I, V,C) \in \bar{\Omega} \,: \,E >0, \,I>0,\,  C >0 \right\},\\
    \partial X_0 &= \left\{ (S, E, I, V,C) \in \bar{\Omega} \,: \,EIC=0 \right\}.
\end{align*}
It suffices to show that system~\eqref{eq:sys} is uniformly persistent with respect to
$(X_0,\partial X_0)$, i.e., there exists $\delta^*>0$ such that
\[
\liminf_{t\to\infty} d\bigl(\Phi_t(x_0),\partial X_0\bigr)\ge \delta^*,
\qquad \text{for all } x_0\in X_0.
\]

Indeed, if $x_0\in X_0$, then the claim is exactly the one established above. If instead $E(0)+I(0)+C(0)>0$ but $x_0\notin X_0$, then necessarily
$x_0\in \partial X_0\setminus\Sigma$, where $\Sigma$ is defined in Lemma~\ref{lem:Mpartial_equals_Sigma}. Hence, by Lemma~\ref{lem:Mpartial_equals_Sigma}, there exists $t_*>0$ such that $\Phi_{t_*}(x_0)\in X_0$. Therefore, once uniform persistence with respect to $(X_0,\partial X_0)$ has been established, the same asymptotic conclusion follows for the trajectory starting from $\Phi_{t_*}(x_0)$, and hence for the original trajectory starting from $x_0$ by the semiflow property. Thus, the result extends to every initial condition satisfying $E(0)+I(0)+C(0)>0$.

To apply the persistence theorem, it remains to characterise the invariant part of the boundary
$\partial X_0$. First, we note that $\partial X_0$ is a compact subset of $\bar{\Omega}$ and that
$X_0$ is positively invariant. We now recall the sets $M_{\partial}$ and $\Sigma$ introduced in
Lemma~\ref{lem:Mpartial_equals_Sigma}:
\[
M_{\partial}
:=\Bigl\{x_0\in \partial X_0:\ \Phi_t(x_0)\in \partial X_0 \ \text{for all } t\ge 0\Bigr\},
\qquad
\Sigma:=\Bigl\{x_0\in \bar{\Omega}:\ E=I=C=0\Bigr\}.
\]
By Lemma~\ref{lem:Mpartial_equals_Sigma}, we have $M_{\partial}=\Sigma$. Restricting
system~\eqref{eq:sys} to $\Sigma$ yields the disease--free subsystem
\begin{equation}\label{eq:DF_subsystem}
\begin{cases}
\dot S=\Lambda-\sigma S+(1-p)\eta V-\mu S,\\[2mm]
\dot V=\sigma S-(\eta+\mu)V,
\end{cases}
\end{equation}
whose unique equilibrium is $(S_0,V_0)$. Since \eqref{eq:DF_subsystem} is a linear--affine system and
$(S_0,V_0)$ is locally asymptotically stable, it is also globally asymptotically stable on $\Sigma$.
Hence,
\[
\Omega_2:=\bigcup_{y\in M_{\partial}}\omega(y)=\{E_0\},
\]
where $E_0=(S_0,0,0,V_0,0)$ is the disease--free equilibrium of~\eqref{eq:sys}.
In particular, $\{E_0\}$ is a compact isolated invariant set (since it is the unique equilibrium on $M_{\partial}$ and it is globally stable) and provides an acyclic isolated covering of
$\Omega_2$ (since there exists no solution on $M_{\partial}$ linking $E_0$ to itself). Finally, by Lemma~\ref{lemma: E0 weak repeller for X0}, $E_0$ is a weak repeller for $X_0$.
Therefore, all assumptions of Theorem~4.5 in~\cite{thieme1993persistence} are satisfied, and $\partial X_0$ is a uniform strong repeller for $X_0$.
Hence, there exists $\delta^*>0$ such that
\[
\liminf_{t\to\infty} d\bigl(\Phi_t(x_0),\partial X_0\bigr)\ge \delta^* \qquad \forall\,x_0\in X_0.
\]
This implies that $E(t)$, $I(t)$ and $C(t)$ are uniformly bounded away from zero for all sufficiently large
times, and in particular
\[
\liminf_{t\to\infty}E(t)\ge \delta^*,\qquad
\liminf_{t\to\infty}I(t)\ge \delta^*,\qquad
\liminf_{t\to\infty}C(t)\ge \delta^*.
\]

\medskip
\noindent
Setting $\delta:=\min\{\delta_S,\delta_V,\delta^*\}>0$ yields the desired uniform persistence estimate for all
components.

Finally, since the semiflow generated by \eqref{eq:sys} on $X=\bar\Omega$ is point dissipative
(Proposition~\ref{lemma: point dissipativity}) and uniformly persistent with respect to
$(X_0,\partial X_0)$, Theorem~2.4 in~\cite{zhao1995uniform} implies the existence of an equilibrium $\bar{x}\in X_0$. Therefore, system~\eqref{eq:sys}
admits at least one endemic equilibrium when $\mathcal{R}_c >1$.
\end{proof}

\section{Concluding remarks}
\label{sec:conclusions}
In this work, we have analysed the global dynamics of an epidemic model incorporating demographic turnover, vaccination with delayed immune maturation, and a saturating fomite--mediated transmission pathway. The environmental component is modelled through a bounded Holling-type functional response, allowing us to distinguish between linear (type~II) and higher-order (type~III) dose--response relationships. The present work complements the local and bifurcation analysis in~\cite{gokcce2024dynamics}
by providing the following global results and threshold characterisations:

\begin{itemize}
\item For a Holling type~II dose--response function, an explicit sufficient condition $\mathcal{J}_c\leq 1$ is derived for the global asymptotic stability of the disease-free equilibrium, by applying the Kamgang--Sallet approach for monotone systems with a Metzler infected subsystem. The threshold $\mathcal{J}_c$ is expressed in closed form and yields a transparent upper bound on infection growth within the invariant region.

\item In the absence of vaccination,  thresholds  $\mathcal{R}_0$ and $\mathcal{J}_c$
coincide, recovering the sharp threshold $\mathcal R_0\le 1$ for the global asymptotic stability of the disease-free equilibrium.

\item When $\mathcal{R}_c>1$, uniform persistence is established by using persistence theory for semiflows with an acyclicity analysis of the boundary dynamics, implying that all infected components remain uniformly bounded away from zero for large times.
\end{itemize}

From an epidemiological perspective, our results quantify the combined impact of vaccination and environmental contamination. Vaccination affects transmission both directly, by reducing susceptibility, and indirectly, by decreasing environmental pathogen shedding through reductions in infectious prevalence. %\cite{chakroborty2026seirv, pathak2025mathematical}.
In particular, the vaccination rate $\sigma$ and the effectiveness parameter $p$ influence the disease--free equilibrium values $S_0$ and $V_0$, and therefore enter the control number $\mathcal R_c$, yielding local stability depending on them. However, the global threshold $\mathcal J_c$ we obtained does not depend on these values. Nevertheless, an increase in vaccination coverage reduces the pool of susceptible individuals and can push the system below the extinction threshold. %\cite{anderson1991infectious,diekmann2010construction}.
The saturation in environmental transmission prevents unrealistically large infection pressure at high contamination levels and may significantly alter the role of fomites in sustaining transmission \cite{liu1986influence,martcheva2015introduction}.
In particular, when environmental transmission is linear at low contamination levels ($n=1$), it directly influences invasion and global stability thresholds.
For higher-order responses ($n \ge 2$), environmental transmission does not affect the invasion threshold but can still shape nonlinear and global dynamics.
%%% \noindent\textbf{Sensitivity considerations.}
The explicit expressions of $\mathcal R_c$ and $\mathcal J_c$ allow for qualitative sensitivity insights.
Parameters that increase the pathogen shedding $\varphi$, decrease the environmental decay $\omega$, or increase direct transmission rates $\beta_1$ and $\beta_2$ enlarge the control reproduction number and therefore enlarge the persistence region.
Conversely, increasing the vaccination rate $\sigma$, the vaccine effectiveness $p$, or the recovery rate $\gamma$ reduces the effective control reproduction number.
In particular, the environmental parameters $\varphi$, $\omega$ and $\kappa$ only affect invasion when $n=1$, but they influence the global dynamics for all $n$, highlighting that control strategies targeting environmental decontamination may be more critical \cite{tien2010multiple,li2009dynamics}. From a mathematical standpoint, the combination of a nonlinear environmental incidence function with vaccination-induced transitions leads to a nontrivial threshold structure in which invasion, extinction, and global stability are governed by distinct but related quantities. The explicit comparison between $\mathcal R_c$ and $\mathcal J_c$ highlights the gap between local and global thresholds in systems with bounded nonlinear feedback and imperfect vaccination.

We acknowledge that this work has some limitations, which nonetheless leave open the possibility of future developments. One direction is the study of backward bifurcation phenomena in the higher-order incidence case. Another direction is the incorporation of waning immunity or booster vaccination strategies, as well as the investigation of time-dependent or seasonally varying parameters. Indeed, several pathogens that can be transmitted via contaminated surfaces exhibit marked seasonality, including both enteric and respiratory ones (e.g., norovirus, influenza, and respiratory syncytial virus) \cite{CDC_NorovirusFactsStats_2024,CDC_FluSeason_2025,CDC_RSVClinicalOverview_2025}. Numerical exploration of parameter sensitivity and quantitative assessment of the relative contribution of environmental transmission would further complement the analytical results.
Overall, the analysis presented here provides rigorous extinction and persistence criteria for epidemic models combining vaccination and saturating fomite--mediated transmission, contributing to the theoretical understanding of environmentally mediated infectious disease dynamics.

\paragraph*{Acknowledgements} This work has been carried out under the auspices of the Italian National Group for Mathematical Physics (GNFM) of the National Institute for Advanced Mathematics (INdAM). E. P. would like to thank Prof. Bruno Buonomo for the fruitful discussions and suggestions during the writing of this paper. B. G. and U. F. gratefully acknowledge financial support from the Robert Bosch Stiftung through the Partnership Scholarship Programme between Johannes Gutenberg University Mainz and the University of Warsaw, which enabled the research collaboration leading to this work.

\paragraph*{Data availability statement} Data sharing not applicable to this article as no datasets were generated or
analysed during the current study.

\section*{Compliance with ethical standards}

\paragraph*{Conflict of interest}
The authors state that there is no conflict of interest.

\appendix

% se non l'hai già nel preambolo:
\numberwithin{equation}{section}
\numberwithin{figure}{section}
\numberwithin{table}{section}

% --- Appendix A ---
\setcounter{section}{1}   % => A
\setcounter{equation}{0}  % riparte da (A.1)
\section*{Appendix}\addcontentsline{toc}{section}{Appendix A}
\subsection*{A. Proof of Proposition \ref{thm: invariance Omega}.}
By standard procedure (see e.g.~\cite{buonomo2025two}), one can derive the positive invariance of the nonnegative cone and that $S(t) >0$ for $t >0$, for any initial data $\mathbf{x}_0 \in \mathbb{R}_+^5$. Summing the balance equations for the compartments $S$, $E$, $I$, $V$ and $R$ gives
\[
\dot N(t)=\Lambda-\mu N(t)-d\,I(t)\ \le\ \Lambda-\mu N(t).
\]
Comparison with $\dot y=\Lambda-\mu y$ yields
\begin{equation*}
N(t)\le N(0)e^{-\mu t}+\frac{\Lambda}{\mu}\left(1-e^{-\mu t}\right),\qquad t\ge 0,
\end{equation*}
hence
\begin{equation*}
\limsup_{t\to\infty}N(t)\le \frac{\Lambda}{\mu}
\quad\text{and}\quad
N(t)\le \max\!\left\{N(0),\,\frac{\Lambda}{\mu}\right\}.
\end{equation*}
Analogously, from $\dot C \leq \varphi N-\omega C$, one can get
\begin{equation*}
\limsup_{t\to\infty}C(t)\le \frac{\varphi \Lambda}{\omega \mu}
\quad\text{and}\quad C(t)\le \max\!\left\{C(0),\,\frac{\varphi \Lambda}{\omega \mu}\right\}.
\end{equation*}
\subsection*{B. Global DFE stability via the Kamgang and Sallet approach}

 \begin{lemma}[Kamgang and Sallet, 2008] \label{thm: kamgang sallet}
        Consider the model in \cref{eq:model_matrices} and suppose that the following assumptions hold:
        \begin{enumerate}[label = \textbf{A\arabic*:}]
        \label{hyp A1}
        \item The model is positively invariant and point dissipative \cite{hale2010asymptotic} on $\Omega$, i.e., there exists a compact set $K \subseteq \Omega$ such that for every $y \in \Omega$, there exists a time $t(y)$ for which $x(t, 0, y) \in \mathring{K}$ for every $t \geq t(y)$.
        \label{hyp A2}
        \item The sub--system $\dot{x}_1 = A_1(x_1^*,0)\cdot (x_1- x_1^*)$ is globally asymptotically stable at the equilibrium $x_1^*$ on $\Omega \cap (\R^{n_1}_+ \times \{0\})$.
         \label{hyp A3}
        \item The matrix $A_2(x)$ is Metzler and irreducible for every $x \in \Omega$.
         \label{hyp A4}
        \item There exist an upper--bound matrix $\Bar{A}_2$ for $\mathcal{M} = \{ A_2(x)\, |\, x \in \Omega \}$ with the property that either $\Bar{A}_2 \notin \mathcal{M}$ or if $\Bar{A}_2 \in \mathcal{M}$, then for any $\Bar{x} \in \Omega$ such that $A_2(\Bar{x}) = \Bar{A}_2$, it holds that $\Bar{x} \in \R^{n_1}_+ \times \{0\}$.
         \label{hyp A5}
        \item The stability modulus $s(\Bar{A}_2) = \max\{ \mathfrak{Re}(\lambda)\, |\, \lambda \in \sigma (\bar{A}_2) \}$ (i.e., the greatest real part of the eigenvalues) of $\Bar{A}_2$ is non--positive.
        \end{enumerate}
        In these hypotheses, the disease--free equilibrium $x^*$ of the system \eqref{eq:model_matrices} is globally asymptotically stable in $\Bar{\Omega}$.
    \end{lemma}
\bibliographystyle{abbrv}
\bibliography{eub}
%%\end{linenumbers}
\end{document}